\documentclass[journal]{IEEEtran}

\usepackage{amsmath}    
\usepackage{graphics,epsfig,subfigure}    
\usepackage{verbatim}   
\usepackage{color}      
\usepackage{subfigure}  
\usepackage{amssymb}
\usepackage{wasysym}
\usepackage{epstopdf}
\usepackage{graphicx}
\usepackage{mathrsfs}

\newcommand{\prob}[1]{\text{Pr}\big\{#1\big\}}
\newcommand{\expect}[1]{\mathbb{E}\big\{#1\big\}}

\newcommand{\bv}[1]{{\boldsymbol{#1} }}
\newcommand{\script}[1]{{{\cal{#1} }}}

\newtheorem{coro}{\textbf{Corollary}}

\newtheorem{theorem}{\textbf{Theorem}}

\newtheorem{definition}{\textbf{Definition}} 

\allowdisplaybreaks 
\begin{document} 

\title{When Backpressure Meets Predictive Scheduling} 
\author{\large{Longbo Huang}, \large{Shaoquan Zhang}, \large{Minghua Chen}, \large{Xin Liu} 
\thanks{Longbo Huang  (http://www.iiis.tsinghua.edu.cn/$\sim$huang) is with the Institute for Theoretical Computer Science and the Institute for Interdisciplinary Information Sciences, Tsinghua University, Beijing, P. R. China. } 
\thanks {Shaoquan Zhang and Minghua Chen (\{zsq008, minghua\}@ie.cuhk.edu.hk) are with the Dept. of Information Engineering at the Chinese University of Hong Kong, Shatin, Hong Kong.}
\thanks{Xin Liu (liuxin@microsoft.com) is with Microsoft Research at Asia, Beijing.}
\markboth{Draft}{Huang} 
} 

\maketitle

\begin{abstract}
Motivated by the increasing popularity of learning and predicting human user behavior in communication and computing systems, in this paper, we investigate the fundamental benefit of predictive scheduling, i.e., predicting and pre-serving arrivals, in controlled queueing systems. Based on a lookahead window prediction model, we first establish a novel equivalence between the predictive queueing system  with a \emph{fully-efficient} scheduling scheme and an equivalent queueing system without prediction. This connection allows us to analytically demonstrate that predictive scheduling necessarily improves system delay performance and can drive it to zero with increasing prediction power.  
We then propose the \textsf{Predictive Backpressure (PBP)} algorithm for achieving optimal utility performance in such predictive systems. \textsf{PBP} efficiently  incorporates prediction into stochastic system control and avoids the great complication due to the exponential state space growth in the prediction window size. 
We show that \textsf{PBP} can achieve a utility performance that is within $O(\epsilon)$ of the optimal, for any $\epsilon>0$, while guaranteeing that the system delay distribution is a \emph{shifted-to-the-left} version of that under the original Backpressure algorithm. Hence, the average packet delay under \textsf{PBP} is strictly better than that under Backpressure, and vanishes with increasing prediction window size. This implies that the resulting utility-delay tradeoff with predictive scheduling beats the known optimal $[O(\epsilon), O(\log(1/\epsilon))]$ tradeoff for systems without prediction. 
\end{abstract} 
\begin{keywords}
Prediction, Queueing, Optimal Control, Backpressure
\end{keywords}

\section{Introduction}
Due to the rapid development of the powerful handheld devices, e.g., smartphones or tablet computers, human users now interact much more easily and frequently with the communication and computing infrastructures, e.g., E-commerce websites, cellular  networks, and crowdsourcing platforms. In this case, in order to provide high level quality-of-service, it is important to understand human behavior features and to use the information in guiding system control algorithm design. Therefore, various studies have been conducted to learn and predict human behavior patterns, e.g., online social networking \cite{maia-social-08}, online searching behavior \cite{weber-searching-11}, and online browsing \cite{kumar-browsing-10}. 

In this paper, we take one step further and ask the following important question: \emph{What is the fundamental benefit of having such user-behavior information?} Our objective is to obtain a \emph{theoretical quantification} of this benefit. To mathematically carry out our investigation, we consider a $N$-user single-server queueing system. At every time slot, user workload arriving at the system will first be queued at the corresponding buffer space. Then, the server allocates resources and decides the scheduling for serving the jobs. The action allows the server to serve certain amount of workload for each user, but also results in a system cost. Different from most existing work in multi-queue system scheduling, here we assume that the server  can \emph{predict and serve the future arrivals before they arrive at the system.} Hence, at every time, the server updates his prediction of the future arrivals and adapts his control action. The objective of the system is to serve all the user workload and to ensure small job latency for each user. 

This is an important problem and can be used to model many practical systems where traffic prediction can be performed. The first example is scheduling in cellular networks. In this case,  the base station handles users' data demand. Instead of waiting for the users to submit their requests, and suffering from a potentially big burst of traffic,  which can lead to a large service latency, the base station can ``push'' the information to the users beforehand, e.g., push the news information at $7$am in the morning.  
The second example scenario is prefetching in computing systems, e.g.,  \cite{padmanabhan-96} \cite{lee-2012}.  Here, data or instructions are preloaded into memory before they are actually requested. Doing so enables faster access or execution of the command and enhances system performance. 
Another example of the model is computing program management in, e.g., computers. In this case, each user represents a software application and the server represents a workload management unit. Then, according to the needs of the applications, the managing unit pre-computes some information in case some later applications request them, e.g., branch prediction in computer architecture \cite{ball-93} \cite{farooq-2013}. 
%

There have been many previous works studying multi-queue system scheduling with utility optimization. \cite{berry-energy} studies the fundamental tradeoff between energy consumption and packet delay for a single-queue system. \cite{neelyenergydelay} further extends the result to a downlink system and designs algorithms to achieve the optimal tradeoff. \cite{neelyenergy} designs algorithms for minimizing energy consumption of a stochastic network. \cite{AFu_Opt_energy_ton03} designs energy optimal scheme for satellites. \cite{zafer-energy} looks at the problem of quality-of-service guaranteed energy efficient transmission using a calculus approach. \cite{tan-downlink-09} studies the tradeoff between energy and robustness for downlink systems. \cite{neelysuperfast} and \cite{huangneely_dr_tac} develop algorithms for achieving the optimal utility-delay tradeoff in multihop networks.  

However, we note that all the aforementioned works assume that the system only takes  \emph{causal} scheduling actions, i.e.,  the server will start to serve the packets only after they arrive at the system. While this is  necessary in many systems,  pre-serving future traffic can be done in systems that have highly predictable traffic. 
While predictive scheduling approaches have been investigated, e.g., \cite{farooq-2013}, not much analytical study has been conducted. 
%
Closest to our work are \cite{tadrous-proactive}, \cite{tadrous-proactive-shaping}, which study the benefit of proactive scheduling, 
and \cite{xu-mm1}, which studies the impact of future arrival information on queueing delay in $M/M/1$ queues.
However, we note that \cite{tadrous-proactive} and \cite{tadrous-proactive-shaping} do not consider the effect of queueing, which very commonly appears in communication and computing systems, whereas \cite{xu-mm1} considers a single $M/M/1$  queue without controlled service rates and scheduling. 
Indeed, due to the joint existence of prediction and controlled queueing, the problem considered here is very complicated. Delay problems for controlled queueing systems are known to be hard. On top of that, arrival prediction advances in a sliding-window pattern over time, i.e., at every time, the system can predict slightly further into the future. Designing control algorithms for such systems often  involves dynamic programming (DP). However, since the state space size grows exponentially with the prediction window size, the DP approach may not be computationally practical in this case even for small systems. Even without prediction, the complexity of DP can still be very high due to the large queue state space. 

To resolve the above difficulty, we first establish a novel equivalence between the queueing system under prediction and a class of \emph{fully-efficient} scheduling scheme and a  queueing system without prediction but with a different initial condition and an equivalent scheduling policy. 
This connection is made by carrying out a sample-path queueing argument and enables us to analytically quantify the delay gain due to predictive scheduling for general multi-queue single-server systems. This result shows that for such systems, the packet delay distribution is \emph{shifted-to-the-left} under predictive scheduling. Hence, the average delay necessarily decreases and \emph{approaches zero} as the prediction window size increases. Based on this result, we further propose a low-complexity \textsf{predictive Backpressure (PBP)} scheduling policy for utility maximization in such predictive systems. \textsf{PBP} retains many features of the original Backpressure algorithm, e.g., greedy, does not require statistical information of the system dynamics, and has strong theoretical performance guarantee.   
We prove that the \textsf{PBP} algorithm can achieve an average cost that is $O(\epsilon)$ of the minimum cost for any $\epsilon>0$, while guaranteeing an average delay that is strictly smaller than that under the original Backpressure. Then, for the case when the first-in-first-out (FIFO) queueing discipline is used, we also explicitly show that \textsf{PBP} achieves an average packet delay of $O(1/\epsilon - D)$ (Here $D$ is the prediction window size). This demonstrates the power of predictive scheduling and provides a rigorous quantification of the benefit. 

The rest of the paper is organized as follows. In Section \ref{section:model}, we present our system model and problem formulation. We then review the known results in the Backpressure literature (without prediction)  in Section \ref{section:review}. Then, we develop the \textsf{Predictive Backpressure (PBP)} algorithm in Section \ref{section:pbp}. The analysis of delay performance under general predictive scheduling and \textsf{PBP} is given in Section \ref{section:analysis}. Simulation results are presented in Section \ref{section:simulation}. We conclude our paper in \ref{section:conclusion}.

\section{System Model}\label{section:model}
We consider a general multi-queue single-server system shown in Fig. \ref{fig:multi-queue}. In this system, a server maintains $N$ queues, one for each user that utilizes the service of the server.  
This multi-queue system has many applications. For instance, it can be used to model  downlink transmission  in cellular networks, where the server represents the base station and the users are mobile users. Another example is the task management system of smartphones, where each user represents an application and 
the server represents the operating system  that manages all computing workloads. %
We assume that the system operates in slotted time, i.e., $t\in\{0, 1, ...\}$. 
\begin{figure}[cht]
\centering
\includegraphics[height=1.2in, width=2.4in]{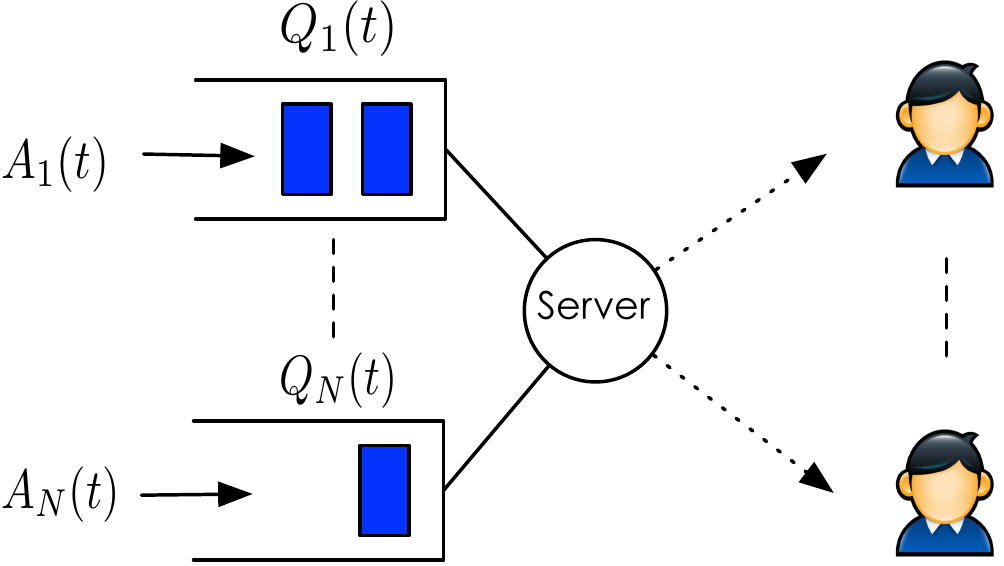}
\vspace{-.05in}
\caption{A multi-queue system where a server is managing workloads for different users/applications.}\label{fig:multi-queue}
\vspace{-.1in}
\end{figure}

\subsection{The traffic model}
We use $A_n(t)$ to denote the amount of new workload arriving to the system at time $t$ (called packets below). Here the workload can represent the newly arrived data units that need to be delivered to their destinations,  
or the new computing tasks that the server must fulfill eventually. 
We use $\bv{A}(t)=(A_1(t), ..., A_N(t))$ to denote the vector of arrivals at time $t$. 
We assume that  $\bv{A}(t)$ is i.i.d. with $\expect{A_n(t)}=\lambda_n$. We also assume that for each $n$, $0\leq A_n(t)\leq A_{\max}$.  

\subsection{Th service rate model} 
Every time slot, the server allocates power for serving the pending packets. \footnote{Note that our results can also be extended to the case where the server also consumes other type of resources, e.g., CPU cycles.} 
However, due to the potential system dynamics, e.g., channel fading coefficient changes, serving different users at different time may result in different resource consumption and generate different service rates. We model this fact by assuming that the server connects to each user $n$ with a time-varying channel, whose state is denoted by $S_n(t)$. We then denote $\bv{S}(t)=(S_1(t), ..., S_N(t))$ the system link state. 
We assume that $\bv{S}(t)$ is i.i.d. and takes values in $\{s_{1}, ..., s_{K} \}$. \footnote{Note that our results can easily be generalized to the case when both the arrivals and the channel conditions are Markovian using the variable-size drift analysis developed in \cite{huangneely_qlamarkovian}.}
We use $\pi_{s_i}$ to denote the probability that $\bv{S}(t) = s_i$. 

The server's power allocation over link $n$ at time $t$ is denoted by $P_n(t)$. We denote the aggregate system power allocation vector by $\bv{P}(t)=(P_1(t), ..., P_N(t))$. Under a system link state $s_i$, we assume that the power allocation vector $\bv{P}(t)$ must be chosen from some feasible power allocation set $\script{P}^{(s_i)}$, which is compact and contains the constraint  $0\leq P_n(t)\leq P_{\max}$. Then, under the given link state $\bv{S}(t)$ and the power allocation vector $\bv{P}(t)$, the amount of backlog that can be served for user $n$ is determined by $\mu_n(t)=\mu_n(\bv{S}(t),\bv{P}(t))$. We assume that $\mu_n(\bv{S}(t),\bv{P}(t))$ is a continuous function of $\bv{P}(t)$ for all $\bv{S}(t)$. Also, we assume that there exists $\mu_{\max}$ such that $\mu_n(\bv{S}(t),\bv{P}(t))\leq\mu_{\max}$ for all time $t$ under any $\bv{S}(t)$ and $\bv{P}(t)$. 
%


\subsection{The predictive service model}
Different from most previous works, we assume that the server can \emph{predict and serve} future packet arrivals. Specifically, we first parameterize our prediction model by a vector $\bv{D}=(D_1, ..., D_N)$, where $D_n$ is the prediction window size of user $n$. That is, at each time $t$, the server has access to the arrival information in the lookahead window $\{A_n(t), ..., A_n(t+D_n-1)\}$, and can allocate rates to serve the future arrivals in the current time slot. \footnote{Since we assume that the arrivals in a time slot can only be served in the next slot, we also consider $A_n(t)$ future arrivals. }
Such a lookahead window model was also used in \cite{xu-mm1} and \cite{lin-igcc-12}. 

We then use $\{\mu_n^{(d)}(t)\}_{d=0}^{D_n-1}$ to denote the rate allocated to serve the arriving packets  in time slot $t+d$ and let $\mu_n^{(-1)}(t)$ to denote the rate allocated for serving the packets that are already in the system. Note that we always have $\sum_{d=-1}^{D_n-1}\mu_n^{(d)}(t)\leq\mu_n(t)$. 
Fig. \ref{fig:slot-prediction} shows the slot structure and the predictive service model. 
\begin{figure}[cht]
\centering
\includegraphics[height=1.2in, width=3.4in]{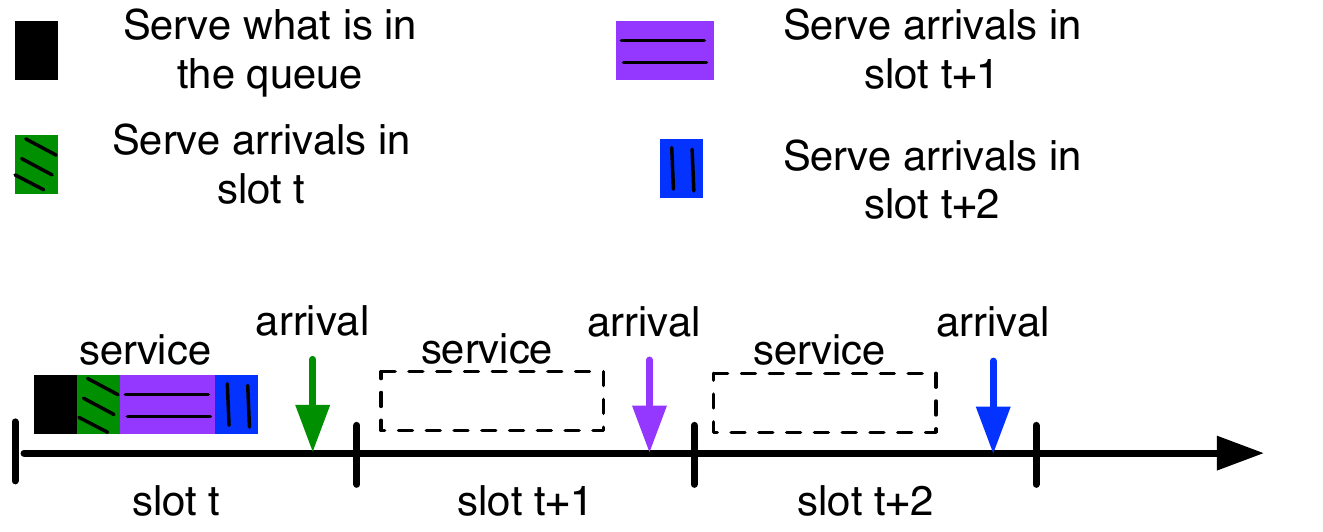}
\vspace{-.05in}
\caption{What happens in a single slot in the case of $D_n=3$. The server predicts the arrivals in slot $t$, $t+1$ and $t+2$. }\label{fig:slot-prediction}
\vspace{-.1in}
\end{figure}

Note that the lookahead window model is an idealized model which assumes  that the system can perfectly predict the future arrivals. Because of this, our results can be viewed as  quantifying the fundamental benefit of predictive scheduling. 

We note that this model is very different from previous controlled queueing system works, which almost all assume that the system only serve the packets in a \emph{causal} manner, i.e., only serve them after they arrive at the system. Our model is  motivated by pre-fetching techniques used in memory management \cite{padmanabhan-96}, branch prediction in computer architecture \cite{ball-93}, as well as recent advancement in data mining for learning user behavior patterns  \cite{kumar-browsing-10}. . 


\subsection{Queueing}
Denote $Q_n(t)$ the number of packets queued at the server for user $n$. We assume the following queueing dynamics: 
\begin{eqnarray}
Q_n(t+1)  = \bigg[Q_n(t) - \mu_n^{(-1)}(t)\bigg]^++A_{n}^{(-1)}(t) . \label{eq:q-dyn}
\end{eqnarray}
Here $A_{n0}(t)$ denotes the number of packets that actually enter the queue after going through a series of predictive service phases, i.e., for all $-1\leq d\leq D_n-2$, 
\begin{eqnarray}
\hspace{-.3in}&&A_{n}^{(d)}(t)  = [A_{n}^{(d+1)}(t) - \mu_n^{(d+1)}(t-d-1)]^+, 
\end{eqnarray}
and $A_{n}^{(D_n-1)}(t)=A_n(t)$. 
%
In this paper, we say that the system is \emph{stable} if the following condition holds: 
\begin{eqnarray}
Q_{\text{av}}\triangleq\limsup_{t\rightarrow\infty}\frac{1}{t}\sum_{\tau=0}^{t-1}\sum_n\expect{Q_n(\tau)} <\infty. \label{eq:stable}
\end{eqnarray}

\subsection{System Objective}
In every time slot, the server spends certain cost due to power expenditure. We denote this cost by $f(\bv{S}(t), \bv{P}(t))$. One simple example is $f(S(t), \bv{P}(t))=\sum_nP_n(t)$, which denotes the total power consumption. We make the mild assumptions that $f(\bv{S}(t), \bv{P}(t))$ is a continuous function of $\bv{P}(t)$ for all $\bv{S}(t)$, and that under any state $\bv{S}(t)$, there exists a constant $f^{\max}$ such that $f(\bv{S}(t), \bv{P}(t))\leq f^{\max}$. The special case when $f(\bv{S}(t), \bv{P}(t))$ is independent of $\bv{P}(t)$ corresponds to the stability scheduling problem \cite{neelynowbook}. 

The system's objective is to find a power allocation and scheduling scheme for minimizing the time average cost, defined as:  
\begin{eqnarray}
f_{\text{av}}\triangleq \limsup_{t\rightarrow\infty}\frac{1}{t}\sum_{\tau=0}^{t-1} \expect{f(\tau)}, 
\end{eqnarray}
subject to the constraint that the queues in the system must be stable, i.e., (\ref{eq:stable}) holds. We then use $f_{\text{av}}^{\bv{D}*}$ to denote the minimum average cost under any feasible predictive scheduling algorithm with prediction vector $\bv{D}$, i.e., those that predict the arrivals for $D_n$ slots and allocates service rates to serving the arrivals within the window $[t, t+D_n-1]$ for each user $n$. We then also use $f_{\text{av}}^{*}$ to denote the minimum average power consumption incurred under any non-predictive scheduling policy. 

\subsection{Discussion of the model}
Our model is most relevant for modeling problems where future workload can be predicted and served before they enter the system. One such application scenario is in cellular networks, where the base station handles users' demand. Since each user typically requests certain news information at specific times, e.g., $7$am in the morning. Instead of waiting for the all the users to submit their requests at the same time, which can lead to a large service latency and high power consumption, one can ``push'' some information to the users beforehand at times when the link condition is good. 

Another important application scenario of the model is computation management in computers or smart mobile devices, e.g., branch prediction  \cite{ball-93}.  In this case, each user represents a software program and the server represents a workload management unit. Then, according to the flow of the program, the managing unit can  pre-execute some instructions in case later computing steps need them. 

Without such predictive control, the cost minimization problem has been extensive studied and algorithms have been proposed, e.g., \cite{neelyenergy}. However, very little is known about the fundamental impact of prediction in system performance, let alone finding optimal control policies for such predictive queueing systems. Moreover, due to the existence of prediction windows and the fact that arrival processes are stochastic, the system naturally evolves according to a Markov chain whose state space size grows exponentially in the prediction window size. Thus, this problem is very challenging to solve. 

\section{Review of Results without Prediction}\label{section:review}
In this section, we first review some known results of the problem and the  known Backpressure algorithm, which does not use prediction, i.e., $D_n=0$ for all $n$. Backpressure has been proven to be a very useful technique for utility maximization problems in stochastic queueing systems \cite{neelynowbook}.
The results in this section will be useful for our later analysis. Readers familiar with the Backpressure literature can skip this section. 

Note that when there is no prediction, $\mu^{(d)}_n(t)=0$ for all $d\geq0$ and $\mu^{(-1)}_n(t)=\mu_n(t)$. The queueing dynamic  equation thus reduces to: 
\begin{eqnarray}
\hspace{-.2in}&&Q_n(t+1) = \big[Q_n(t)  - \mu_n(t)\big]^+ + A_n(t)\,\,\forall\,\,n. \label{eq:reduced-q}
\end{eqnarray}
In this case, the following theorem from \cite{neelyenergy} that characterizes $f_{\text{av}}^{*}$ in the non-predictive case. 
\begin{theorem}
The minimum average cost $f_{\text{av}}^{*}$ is the solution to the following nonlinear optimization problem: 
\begin{eqnarray}
\hspace{-.4in}&&\min: \,\, f_{\text{av}} = \sum_{s_i}\pi_{s_i}\sum_{m=1}^{N+2}a_m^{(s_i)}f(s_i, \bv{P}_m^{(s_i)})\label{eq:opt-noprediction}\\
\hspace{-.4in}&&\,\,\,\text{s.t.} \quad  \sum_{s_i}\pi_{s_i}\sum_{m=1}^{N+2}a_m^{(s_i)}\mu_n(s_i, \bv{P}_m^{(s_i)})\geq \lambda_n, \forall\, n,\label{eq:cond}\\
\hspace{-.4in}  &&\qquad\quad \bv{P}_m^{(s_i)} \in\script{P}^{(s_i)}, \,\forall\, s_i, m,\nonumber\\
\hspace{-.4in}  &&\qquad\quad \sum_{m}a_m^{(s_i)}=1, \, a_m^{(s_i)}\geq0, \, \forall\, s_i, m. \quad\Diamond\nonumber
\end{eqnarray}
\end{theorem}
The Backpressure algorithm without prediction works as follows  \cite{neelyenergy}: 

\underline{\textsf{Backpressure (BP):}} In every time slot, observe $\bv{Q}(t)=(Q_1(t), ..., Q_N(t))$ and the current channel state vector $\bv{S}(t)$. Do: 
\begin{itemize}
\item Choose the power allocation vector to solve the following problem: 
\begin{eqnarray}
\hspace{-.3in}\min: && V f(\bv{P}(t))  - \sum_nQ_n(t) \mu_n(\bv{S}(t), \bv{P}(t))\\
\hspace{-.3in}\text{s.t.} && \bv{P}(t) \in \script{P}^{\bv{S}(t)}. 
\end{eqnarray}
\item Update the queue sizes according to (\ref{eq:reduced-q}). $\Diamond$
\end{itemize}
Note that the \textsf{BP} algorithm does not require any statistical information of the arrival and the channel states. The performance of \textsf{BP} has been extensitvely analyzed and the following theorem is also from \cite{neelyenergy}. 
\begin{theorem}\label{theorem:bp}
Suppose there exist a set of power vectors and probabilities $\{\bv{P}_m^{(s_i)}, a_m^{(s_i)}\}$, and a constant $\eta>0$, such that: 
\begin{eqnarray}
\lambda_n - \sum_{s_i}\pi_{s_i}\sum_{m=0}^{\infty}a_m^{(s_i)}\mu_n(s_i, \bv{P}_m^{(s_i)}) \leq -\eta, \,\,\forall\,n. \label{eq:slack}
\end{eqnarray}
Then, the \textsf{BP} algorithm with any finite $\bv{Q}(0)$ achieves the following: 
\begin{eqnarray}
f_{\text{av}}^{\textsf{BP}} \leq f_{\text{av}}^{*} +\frac{B}{V},\quad 
Q_{\text{av}}^{\text{BP}} \leq  \frac{B+Vf^{\max}}{\eta}.  
\end{eqnarray}
Here $B = \frac{N}{2}(\mu_{\max}^2 + A_{\max}^2)$ is a constant independent of $V$. $f_{\text{av}}^{\textsf{BP}}$ and $Q_{\text{av}}^{\textsf{BP}}$ denote the average cost  and the average system queue size under Backpressure, respectively. $\Diamond$
\end{theorem}

The condition (\ref{eq:slack}) is important. It is known that having such an $\eta>0$ is a sufficient condition for stability, and it is necessary to have $\eta\geq0$ \cite{neelynowbook}. 
Throughout this paper, we assume that (\ref{eq:slack}) holds with $\eta>0$.

\section{Predictive Backpressure}\label{section:pbp}
In this section, we describe how prediction can be incorporated into Backpressure and achieve significant delay improvement. 
Since the future arrival information is made available in a sliding-window form,  prediction couples the current action with the future arrivals in every time slot. This prohibits the use of frame-based Lyapunov technique \cite{huangneely_qlamarkovian},  
and makes the problem very challenging. However, as we will see, with the development of a novel sample-path queueing equivalence result, one can incorporate prediction into system control in a very clean manner. 

\vspace{-.1in}
\subsection{Prediction Queues}
To facilitate our analysis, we first introduce the notion of a prediction queue, which records the number of residual arrivals in every slot in $[t, t+D_n-1]$. 
Specifically, we denote  $\{Q_n^{(d)}(t)\}_{d=0}^{D_n-1}$ the number of remaining arrivals currently in future slot $t+d$, i.e., $d$ slots into the future, and denote $Q_n^{(-1)}(t)$ the number of packets in the system. We see that they evolve according to the following dynamics: 
\begin{enumerate}
\item If $d=D_n-1$, then: 
\begin{eqnarray}
\hspace{-.3in}Q_n^{(d)}(t+1) =  A_n(t+D_n). \label{eq:q-1}
\end{eqnarray}
\item Now if $0\leq d\leq D_n-2$, then: 
\begin{eqnarray}
\hspace{-.3in}Q_n^{(d)}(t+1) = \big[Q_n^{(d+1)}(t) - \mu_n^{(d)}(t)\big]^+. \label{eq:q-2}
\end{eqnarray}
\item For $Q_n^{(-1)}(t)$, we have: 
\begin{eqnarray}
\hspace{-.5in}&&Q_n^{(-1)}(t+1) \label{eq:q-3}\\
\hspace{-.5in}&&\qquad\quad = \bigg[Q_n^{(-1)}(t)- \mu_n^{(-1)}(t)\bigg]^+ +\big[Q_n^{(0)}(t) - \mu_n^{(0)}(t)\big]^+, \nonumber
\end{eqnarray}
with $Q_n^{(-1)}(0)=0$. 
\end{enumerate}
Fig. \ref{fig:queue} shows the definition of the prediction queues. One can see that the queues $Q_n^{(d)}(t)$ for $0\leq d\leq D_n-1$ are not real queues. They simply record the residual arrivals going through the timeline, whereas $Q^{(-1)}_n(t)$ records the true backlog in the system. Notice that $Q^{(-1)}_n(t)$ is exactly the same as $Q_n(t)$ in (\ref{eq:q-dyn}).  Since $Q^{(-1)}_n(t)$ is the only actual queue,  the system is stable if and only if $Q_n^{(-1)}(t)$ is stable. 
\begin{figure}[cht]
\centering
\includegraphics[height=0.7in, width=3.4in]{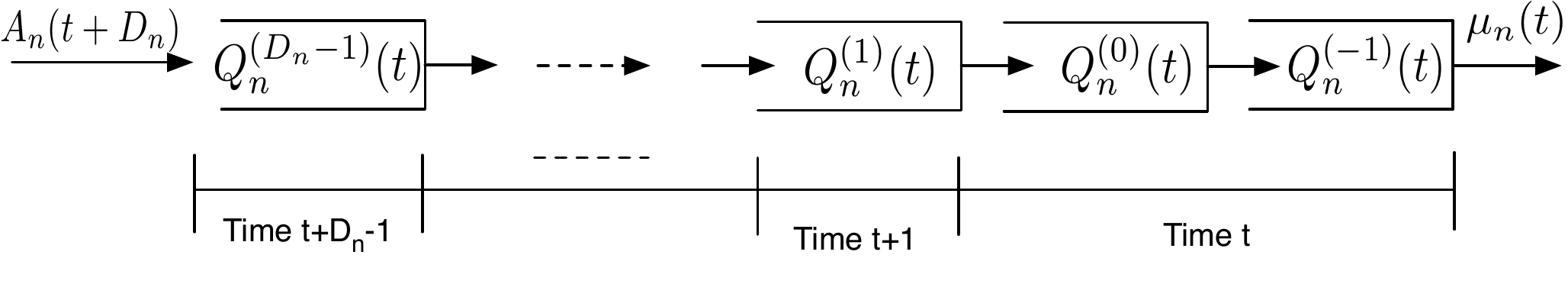}
\caption{The prediction queues that describe the system evolution.}\label{fig:queue}
\end{figure}


\vspace{-.1in}
\subsection{Backpressure with Prediction Queues}
Here we construct our algorithm based on the above prediction queues. The main idea is to use the sum of all the queues $Q_n^{\text{sum}}\triangleq\sum_{d=-1}^{D_n-1}Q_n^{(d)}(t)$ for decision making. 

To describe the algorithm in details, we also define the notion of queueing discipline for the predictive system, i.e., how to select packets to serve from $\{Q_n^{(d)}(t)\}_{d=-1}^{D_n-1}$. Specifically, we first order the packets in $Q_n^{(d)}(t)$ with labels $p_{\sum_{d'<d}Q_n^{d'}(t)+1}, ..., p_{\sum_{d'\leq d}Q_n^{d'}(t)}$. Then, all the packets in $\{Q_n^{(d)}(t)\}_{d=-1}^{D_n-1}$ are ordered from $p_1$ to $p_{Q_n^{\text{sum}}(t)}$. 
When a particular queueing discipline is applied in the predictive system, we select packets to serve according to the discipline using the order of the packets. For instance, if FIFO is used, then the server will serve the packets $p_1, ..., p_{\min[\mu^{(-1)}_n(t), Q_n^{(-1)}(t)]}$ from queue $Q_n^{(-1)}(t)$ and similarly for other queues. We now also define the notion of a \emph{fully-efficient} predictive scheduling policy.  
\begin{definition} 
A predictive scheduling policy is called  \emph{fully-efficient} if for every user $n$, we have: (i) $\sum_d\mu_n^{(d)}(t)=\mu_n(t)$, and (ii) whenever there exists a $d$ such that: 
\begin{eqnarray}
\mu_n^{(d)}(t)> Q_n^{(d)}(t),  
\end{eqnarray}
then, 
\begin{eqnarray}
\mu_n^{d'}(t) \geq Q_n^{d'}(t), \,\,\forall\, d'\neq d.\quad \Diamond
\end{eqnarray}
\end{definition}
In other words, if a policy is  fully-efficient, it will always try to utilize all the service opportunities and not allocate more service rate to serve any  queue unless all other queues are already fully served. \footnote{Note that it is equivalent to work-conserving in queue scheduling.} Hence, it will not waste any service opportunity unless there are more. With this definition, we now present our algorithm. 

\underline{\textsf{Predictive Backpressure (PBP):}} In every time slot, compute $Q_n^{\text{sum}}(t) = \sum_{d=-1}^{D_n-1}Q_n^{(d)}(t)$ for all $n$. Then, observe the current channel state vector $\bv{S}(t)$ and perform: 
\begin{itemize} 
\item Choose the power allocation vector $\bv{P}(t)$ to solve the following problem: 
\begin{eqnarray}
\hspace{-.3in}\min: && V f(\bv{P}(t))  - \sum_nQ_n^{\text{sum}}(t) \mu_n(\bv{S}(t), \bv{P}(t))\\
\hspace{-.3in}\text{s.t.} && \bv{P}(t) \in \script{P}^{\bv{S}(t)}. 
\end{eqnarray}
Then, allocate the service rates $\{\mu_n^{(d)}(t)\}_{d=-1}^{D_n-1}$ to the queues in a fully-efficient manner  according to any pre-specified queueing discipline.  
\item Update the queues according to (\ref{eq:q-1}),  (\ref{eq:q-2}) and  (\ref{eq:q-3}). $\Diamond$
\end{itemize}
Notice that the \textsf{PBP} algorithm has a very clean format and looks very similar to the original Backpressure algorithm. We will show that \textsf{PBP} dramatically reduces the queueing delay compared to the original Backpressure algorithm. 


\section{Performance Analysis}\label{section:analysis}
In this section, we analyze the performance of the \textsf{PBP} algorithm. 
We will first present an important theorem which states that if a predictive scheduling policy is fully-efficient, then the queueing system under the scheme evolves in the  exact  same way  as a non-predictive queueing system with delayed arrivals. Using this result, we obtain an interesting delay distribution shifting theorem. 
After that, we present our delay analysis for the \textsf{PBP} algorithm. 

\subsection{Performance of fully-efficient scheduling policies}
We start by presenting the following  theorem regarding the equivalence between predictive and non-predictive systems. 
\begin{theorem}\label{theorem:sample-path}
Let  $\hat{Q}_n(t)$ be the queue size of a single queue system that (i) has $\hat{Q}_n(0)=\sum_{t=0}^{D_n-1}A_n(t)$, (ii) uses no prediction, (iii) has arrival $\hat{A}_n(t)=A_n(t+D_n)$, (iv) has service $\hat{\mu}_n(t)=\sum_{d=-1}^{D_n-1}\mu_n^{(d)}(t)$, and (v) evolves according to: 
\begin{eqnarray}
\hspace{-.4in}&&\hat{Q}_n(t+1) = \bigg[\hat{Q}_n(t) - \hat{\mu}_n(t) \bigg]^++ \hat{A}_n(t). \label{eq:q-4}
\end{eqnarray}
Then, if the predictive system uses a fully-efficient predictive scheduling policy (with any queueing discipline), we have for all $t$ that: 
\begin{eqnarray}
Q_n^{\text{sum}}(t) = \hat{Q}_n(t), \,\,\,\forall\,\,n. \quad\Diamond. \label{eq:q-equal-thm}
\end{eqnarray}
\end{theorem}
\begin{proof}
See Appendix A. 
\end{proof}
Theorem \ref{theorem:sample-path} provides an important connection between the predictive system and the system without prediction. Using this result, we derive the following theorem, which relates the delay distribution of the predictive system to the equivalent system without prediction. 
%
\begin{theorem} (Delay Distribution Shifting)\label{theorem:dist-shift} 
Denote $\pi^{(D_n)}_{nk}$ the steady-state probability that a user $n$ packet experiences a delay of $k$ slots  under the a fully-efficient predictive scheduling policy in the predictive system, and let $\hat{\pi}_{nk}$ denote the steady-state probability that a user $n$ packet experiences a $k$-slot delay in $\hat{Q}_n(t)$. Suppose the set of queues $\{Q_n^{(d)}(t)\}_{d=-1}^{D_n-1}$ and $\hat{Q}_n(t)$ both use the same queueing discipline. Then, for each queue $n$, we have: 
\begin{eqnarray}
\pi^{(D_n)}_{n0} &=& \sum_{k=0}^{D_n}\hat{\pi}_{nk}, \\
\pi^{(D_n)}_{nk} &=& \hat{\pi}_{nk+D_n}, \,\,k\geq 1.  
\end{eqnarray}
That is, the distribution of the original queue can be viewed as \emph{shifted to the left by $D_n$ slots} under predictive scheduling with $D_n$-slot prediction. $\Diamond$ 
\end{theorem}
\begin{proof}
See Appendix B. 
\end{proof}
Note that Theorem \ref{theorem:dist-shift} is very important to the general framework of predictive scheduling. It allows us to compare scheduling with prediction to the original queueing system. 
To formalize this,  first notice that if we start with $\hat{Q}_n(0)=0$ and have $\hat{A}_n(t)=A_n(t)$, then $\hat{Q}_n(t)$ becomes exactly the same as the queueing process in the \emph{original} system \emph{without} prediction. Thus, if the steady-state behavior of  $\hat{Q}_n(t)$ does not depend on the initial condition and the shift of the arrival process, then the delay performance of the predictive system heavily depends on the delay distribution of the original system without prediction. 
%
\begin{coro}\label{coro:dist-shift} 
Suppose $\{Q_n^{(d)}(t)\}_{d=-1}^{D_n-1}$ and $\hat{Q}_n(t)$ use the same queueing discipline. For any arrival and service processes under which the delay distribution of $\hat{Q}_n(t)$ does not depend on $\hat{Q}_n(0)$ and the delay in the arrival process, we have: 
\begin{eqnarray}
\pi_{n0}^{(D_n)} &=& \sum_{k=0}^{D_n}\pi_{nk}, \\
\pi_{nk}^{(D_n)} &=& \pi_{nk+D_n}, \,\,k\geq 1. 
\end{eqnarray}
Here $\pi_{nk}$ is the steady-state probability that a user $n$ packet experiences a delay of $k$ slots in our system with $D_n=0$.  $\Diamond$ 
\end{coro}
Note that Corollary \ref{coro:dist-shift} applies to general multi-queue single-server systems where the steady-state behavior depends only on the statistical behavior of the arrivals. 

With Theorem \ref{theorem:dist-shift}, we can now quantify how much delay improvement  one obtains via predictive scheduling. 
This is summarized in the following theorem, in which we use $W_{\text{tot}}$ to denote the average delay of the original system without prediction, i.e., 
\begin{eqnarray}
W_{\text{tot}} = \frac{\sum_{n=1}^N\lambda_n \sum_{k\geq0}k\pi_{nk}}{\sum_{n=1}^N\lambda_n }. \label{eq:w-original}
\end{eqnarray}

\begin{theorem}\label{theorem:delay-red}
Suppose the conditions in Corollary \ref{coro:dist-shift} hold. The delay reduction offered by predictive scheduling with prediction window vector $\bv{D}$, denoted by $R(\bv{D})$, is given by: 
\begin{eqnarray}
\hspace{-.1in}R(\bv{D})= \frac{\sum_{n=1}^N\lambda_n\big(  \sum_{1\leq k\leq D_n} k \pi_{nk} + D_n \sum_{k\geq1} \pi_{nk+D_n}  \big)}{ \sum_{n=1}^N\lambda_n }. \label{eq:delay-reduction}
\end{eqnarray}
In particular, if $W_{\text{tot}}<\infty$, the average system delay goes to zero as $D_n$ goes to infinity for all queue $n$, i.e., 
\begin{eqnarray}
\lim_{\bv{D}\rightarrow\infty}\sum_{n}R(\bv{D}) = W_{\text{tot}}.  \label{eq:limit-r}
\end{eqnarray}
Here $\bv{D}\rightarrow\infty$ means that $D_n\rightarrow\infty$ for all $n$. $\Diamond$  
\end{theorem}
\begin{proof}
See Appendix C. 
\end{proof}

We note that Theorems \ref{theorem:dist-shift} and \ref{theorem:delay-red}, 
and Corollary \ref{coro:dist-shift}  show that \emph{systems with predictive scheduling  can be analyzed by studying the original system without prediction}. From (\ref{eq:delay-reduction}), we also see that prediction reduces system delay, even when it is only applied to a subset of queues. This result makes rigorous the intuition that  applying prediction at any queue  effectively saves service opportunities for others and hence reduces system delay. 
Also note that the above results hold  under \emph{any} queueing discipline. The difference is that under different queueing disciplines, the delay distribution of the packets will be different. Hence, the delay reduction due to predictive scheduling will also be different. 



\subsection{Performance of \textsf{PBP}}
In this section, we analyze the performance of \textsf{PBP}. We have the following theorem, which states that allowing predictive scheduling does not change the optimal average cost.  
\begin{theorem}\label{theorem:opt-remains}
For any vector $\bv{D}\succeq\bv{0}$, we have:
\begin{eqnarray}
f_{\text{av}}^{\bv{D}*} &=& f_{\text{av}}^{*}. \quad\Diamond\label{eq:opt-remains}
\end{eqnarray}
\end{theorem}
\begin{proof}
See Appendix D. 
\end{proof}
We see that this is quite intuitive. Since no matter what prediction we have and how one allocates power for serving future arrivals, the minimum consumption is constrained by the average traffic arrival rate. This theorem delivers an important message that predictive scheduling does not change the optimal utility of the system, rather, \emph{it improves the system delay given the same utility performance}. 

We now have the following theorem, which shows that \textsf{PBP} achieves an average power consumption that is with $O(1/V)$ of the minimum and guarantees an average congestion bound. Note that this theorem is very similar to the results in previous literature of Backpressure, e.g., \cite{neelynowbook}, with the difference that here the average queue size combines the actual system backlog and the average size of the prediction queues. 
\begin{theorem}\label{theorem:pbp}
The \textsf{PBP} algorithm achieves the following: 
\begin{eqnarray}
f_{\text{av}}^{\textsf{PBP}} &\leq& f_{\text{av}}^{\bv{D}*} +\frac{B}{V},\label{eq:cost}\\
Q^{\text{sum}}_{\text{av}} & = & Q^{\text{BP}}_{\text{av}} = O(V). \label{eq:qsize}
\end{eqnarray}
Here $Q^{\text{sum}}_{\text{av}}$ denotes the average queue size of $\sum_nQ_n^{\text{sum}}(t)$, and $Q^{\text{BP}}_{\text{av}}$ denotes the average queue size under the Backpressure algorithm without prediction. 
$\Diamond$ 
\end{theorem}
\begin{proof}
See Appendix E. 
\end{proof}

Theorem \ref{theorem:pbp} states that the average size of $\sum_nQ_n^{\text{sum}}(t)$ is the same as $Q^{\text{BP}}_{\text{av}}$ under Backpressure. Since $Q_n^{\text{sum}}(t)$ is the total size of the actual queue and the prediction queues, we can already see that the actual queue size will be smaller than that under the original Backpressure scheme. Since the average queue size under \textsf{PBP} is also finite, we can  use Theorem \ref{theorem:delay-red} to have the following immediate corollary. 
\begin{coro} 
Suppose there exists a steady-state distribution of the queue vector under  \textsf{PBP}. Then, the average delay under \textsf{PBP} goes to zero as $\bv{D}\rightarrow\infty$. $\Diamond$
\end{coro}

It is tempting to analyze the exact delay reduction offered by performing predictive scheduling in Backpressure. 
However, due to the complex queueing dynamics under Backpressure, it is  challenging to compute the exact distributions $\pi_{nk}$ even without prediction and  to obtain explicit delay reduction results under general settings. 
Therefore, in below, we will that for a general class of cost-minimization problems, predictive scheduling improves the average system delay by an amount that is linear in the prediction window size. 

To do so, we will use a theorem from \cite{huangneely_dr_tac}, which shows that under Backpressure, the queue vector of the system is typically ``attracted'' to a fixed point of size $\Theta(V)$. For stating the algorithm, we define the following dual problem of a scaled version of (\ref{eq:opt-noprediction}). 
\begin{eqnarray}
\max: \quad g(\bv{\gamma}),  \quad\text{s.t.}\,\,\,\bv{\gamma}\succeq\bv{0}, \label{eq:dual}
\end{eqnarray}
where $g(\bv{\gamma})$ is the dual function defined as: 
\begin{eqnarray}
\hspace{-.3in}&&g(\bv{\gamma}) = \sum_{s_i}\pi_{s_i}\inf_{  \bv{P}_m^{(s_i)} \in\script{P}^{(s_i)} }\bigg\{  Vf(s_i, \bv{P}_m^{(s_i)}) \label{eq:dual-fun}\\
\hspace{-.3in}&&\qquad\qquad\qquad\qquad\qquad+ \sum_n\gamma_n[  \lambda _n -\mu_n(s_i, \bv{P}_m^{(s_i)}) ]  \bigg\}. \nonumber
\end{eqnarray}
Notice that $g(\bv{\gamma})$ is the dual function of (\ref{eq:opt-noprediction}) with a scaled objective (by $V$) and without the variables $a_m^{(s_i)}$. Now we have the following theorem (which is Theorem 1 in \cite{huangneely_dr_tac}), in which we use $\bv{\gamma}^*$ to denote the optimal solution of (\ref{eq:dual}). According to \cite{huangneely_dr_tac}, we know that $\bv{\gamma}^*$ is either $\Theta(V)$ or $0$. 
\begin{theorem}\label{thm:prob_multi_con}
Suppose (i) $\bv{\gamma}^*$ is unique, (ii) the $\eta$-slack condition (\ref{eq:slack}) is satisfied with $\eta>0$, (iii) the dual function $g(\bv{\gamma})$ satisfies: 
\begin{eqnarray}
g(\bv{\gamma}^*)\geq g(\bv{\gamma})+L||\bv{\gamma}^*-\bv{\gamma}||,\quad\forall\,\,\bv{\gamma}\succeq\bv{0}\label{eq:dualpolyhedralcond},
\end{eqnarray} 
for some constant $L>0$ independent of $V$. Then, under Backpressure without prediction, there exist constants $G, K, c=\Theta(1)$, i.e., all independent of $V$, such that for any $m\in \mathbb{R}_+$, 
\begin{eqnarray}
\script{P}^{(r)}(G, Km)&\leq& ce^{-m},\label{eq:prob_pmr_special}
\end{eqnarray}
where $\script{P}^{(r)}(G, Km)$ is defined:
\begin{eqnarray}
\hspace{-.3in}&&\script{P}^{(r)}(G, Km)\label{eq:pmr_def}\\
\hspace{-.3in}&& \quad\triangleq\limsup_{t\rightarrow\infty}\frac{1}{t}\sum_{\tau=0}^{t-1}\prob{\exists\, n, |Q_n(\tau)-\gamma_{n}^*|>G+Km}.\,\,\Diamond\nonumber
\end{eqnarray}
\end{theorem}
\begin{proof}
See \cite{huangneely_dr_tac}. 
\end{proof}
As shown in \cite{huangneely_dr_tac}, the above conditions (i)-(iii) are satisfied in many practical network optimization problems, especially when $\script{P}^{(s_i)}$ are finite. Under this assumptions, Theorem \ref{thm:prob_multi_con} states that the queue vector $\bv{Q}(t)=(Q_1(t), ..., Q_N(t))$ mostly stays around the fixed point $\bv{\gamma}^*$ \cite{huangneely_dr_tac}. 

We now state our theorem regarding the average backlog reduction due to predictive scheduling. 
\begin{theorem} \label{theorem:pbp-delay}
Suppose (i) the conditions in Theorem \ref{thm:prob_multi_con} hold, (ii) $\bv{\gamma}^*=\Theta(V)>0$, (iii) there exists a steady-state distribution of $\bv{Q}(t)$ under \textsf{PBP},  (iv) $D_n=O(\frac{1}{A_{\max}}[\gamma_{n}^*-G-K(\log(V))^2 - \mu_{\max}]^+)$ for all $n$, and (iv) FIFO is used. Then, under \textsf{PBP} with a sufficiently large $V$, we have:  
\begin{eqnarray}
Q_{\text{av}}^{(0)} \leq Q_{\text{av}}^{\text{BP}} - \sum_nD_n[\lambda_n-O(\frac{1}{V^{\log(V)}})]^+. \,\, \Diamond\label{eq:q-av-0}
\end{eqnarray}
\end{theorem}
\begin{proof}
See Appendix F. 
\end{proof}
Using Little's theorem, we see that Theorem \ref{theorem:pbp-delay} implies that the system delay is reduced roughly linearly in the prediction window size $D_n$.  Note that if the $D_n$ value is larger than $O(\frac{1}{A_{\max}}[\gamma_{Vn}^*-G-K(\log(V))^2 - \mu_{\max}]^+)$, the decrement of the system delay will start to become sub-linear. In this case, most packets will not even enter $Q_n^{(-1)}(t)$, resulting in a very small system delay  (see the simulation section for numerical examples).


\section{Simulation}\label{section:simulation}
We present simulation results of the \textsf{PBP} algorithm in this section. We simulation the system in Fig. \ref{fig:multi-queue} with $N=2$. The arrival processes $A_1(t)$ and $A_2(t)$ are independent processes. $A_1(t)$  takes value $3$ or $0$ with probabilities $0.2$ and $0.8$, respectively. $A_2(t)$ takes value $2$ or $0$ with equal probabilities. For both $n=1,2$, we assume that $S_n(t)\in\{1, 2\}$. Then, the service rate is given by $\mu_n(t)=\lfloor\log(1+S_n(t)P_n(t))\rfloor$. We assume that $P_n(t) \in\script{P}^{(S_n)}\triangleq \{0, 5, 10\}$. However, we assume that at any time, only one channel receives nonzero power allocation, i.e., $P_1(t)P_2(t)=0$. 

We set $\bv{D}=\rho*[5, 10]$. Then, we simulate the cases $\rho\in\{1, 3, 5, 10\}$ to see the effect of the predictive scheduling. We simulate the algorithm with $V\in\{1, 3, 5, 10, 20, 50\}$. Each simulation is run for $5\times10^5$ time slots. Fig. \ref{fig:pbp-per} shows the performance of the  \textsf{PBP} algorithm with FIFO. We see from the left plot that the average power consumption decreases as one increases the $V$ value. Indeed, we observe that when $V=5$, the power performance is already very close to the optimal value. The right plot shows the average backlog under  \textsf{PBP}. It is not hard to see that the average system backlog scales as $O(V)$. One also sees that as the prediction window sizes increase, the network delay  decreases linearly in $\bv{D}$. 
\begin{figure}[cht]
\centering
\includegraphics[height=2.3in, width=3.4in]{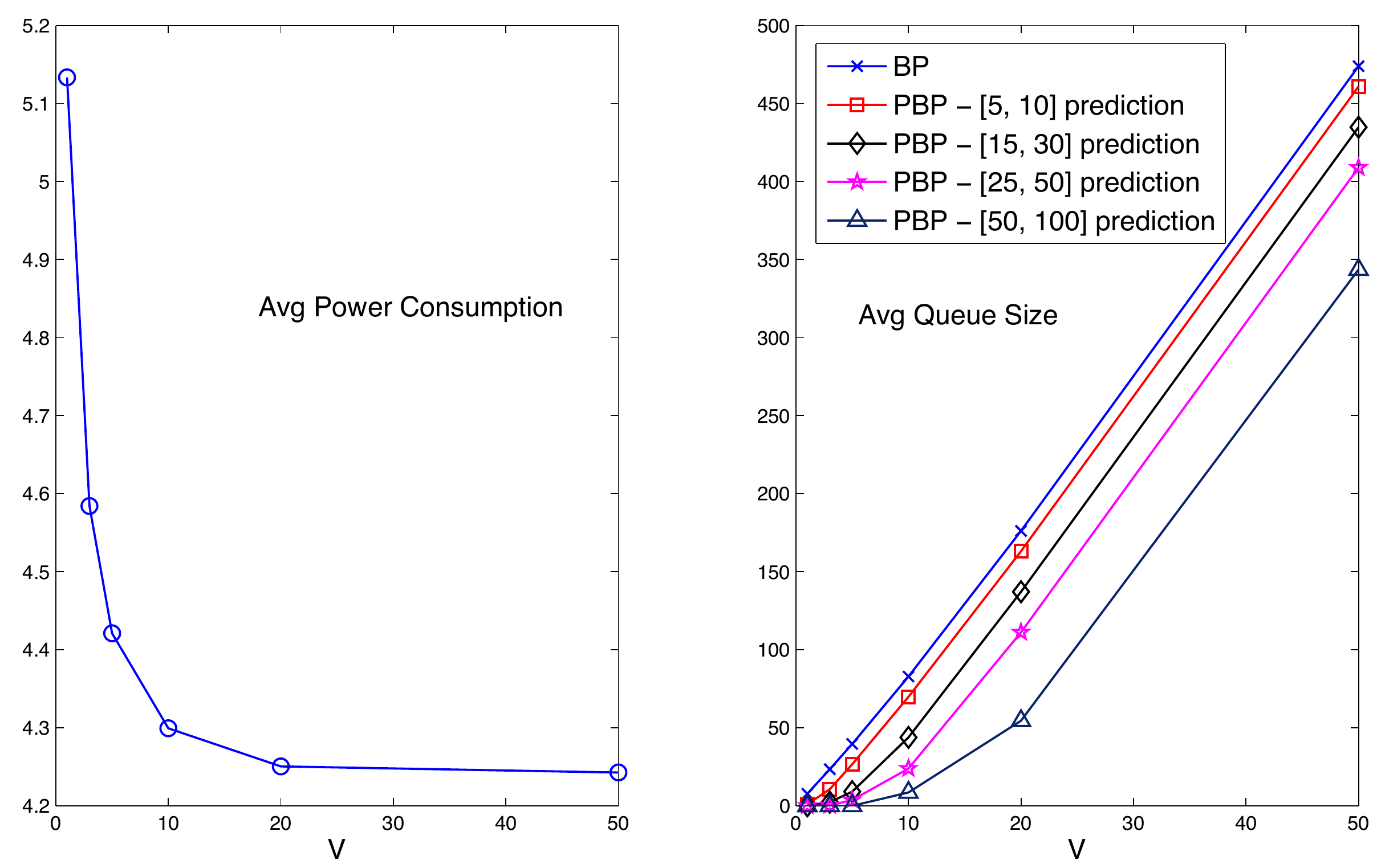}
\vspace{-.1in}
\caption{Left: Average power consumption under \textsf{PBP}. Right: Average queue size under  \textsf{PBP} with different prediction window sizes. }\label{fig:pbp-per}
\vspace{-.1in}
\end{figure}

We now look at the delay distribution of the packets under  \textsf{PBP}. Fig. \ref{fig:dist-fifo} shows the distribution for the setting with $V=10$ and $\rho=3$. We see that  the distributions of the latency for both queues are shifted to the left by $D_1$ and $D_2$, as shown in  Theorem \ref{theorem:dist-shift}. It can also be easily verified that in this case, Corollary \ref{coro:dist-shift} also holds. 
\begin{figure}[cht]
\centering
\includegraphics[height=2.3in, width=3.4in]{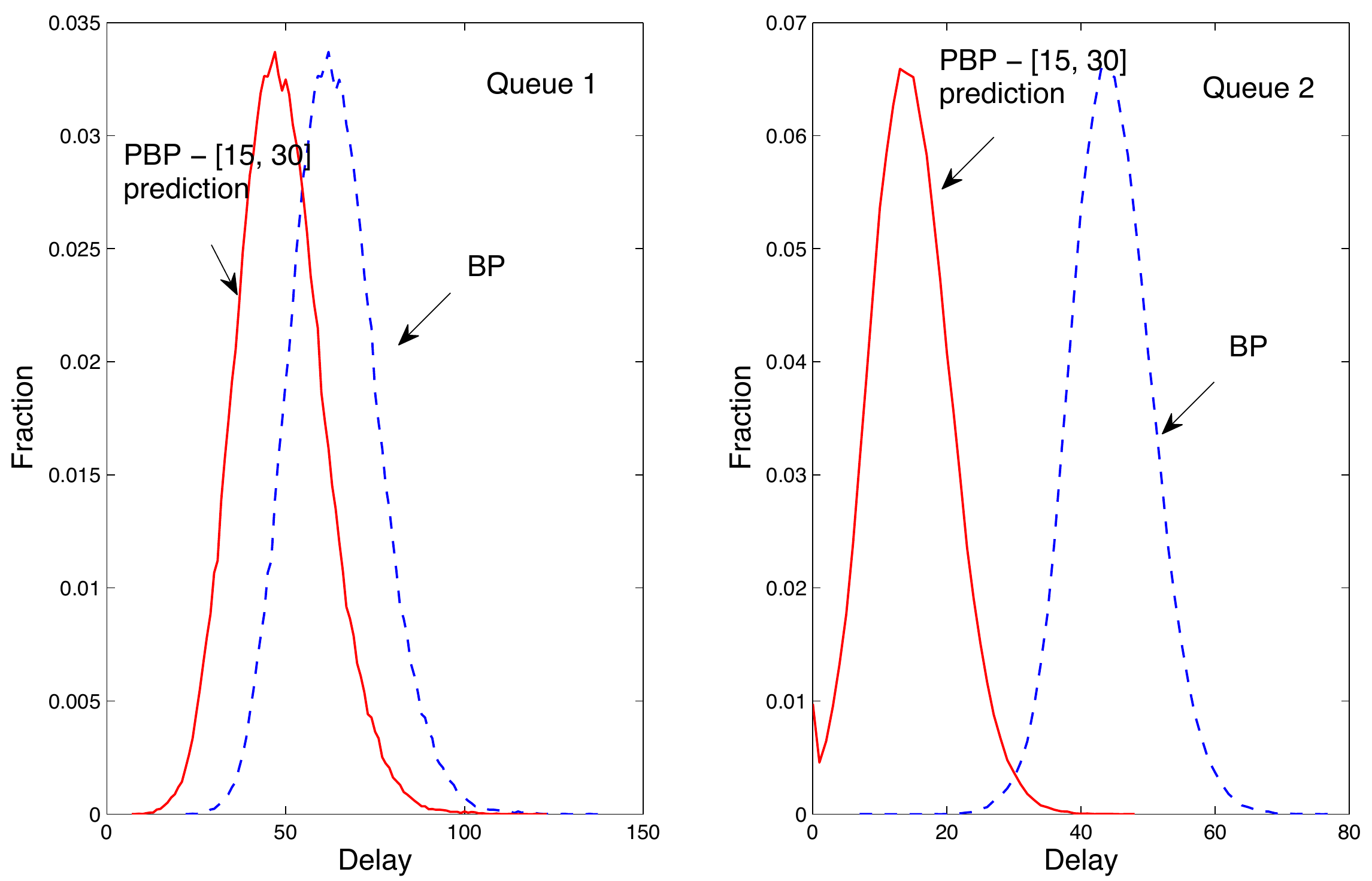}
\vspace{-.1in}
\caption{Packet delay distribution under \textsf{PBP} with FIFO scheduling with $V=10$ and $\bv{D}=[15, 30]$. We see that predictive scheduling effectively shifts the distribution to the left by $D_n$ for queue $n$. }\label{fig:dist-fifo}
\vspace{-.1in}
\end{figure}

Fig. \ref{fig:dist-lifo} then shows the delay distributions under \textsf{PBP} and the original Backpressure, under the LIFO discipline.  It can be verified that the distribution for the predictive system is also a left-shifted version of the one under Backpressure. We see that large fractions of packets experience $0$ delay in both queues, i.e., they are served before they arrive at $Q_n^{(-1)}(t)$. This is so because under LIFO Backpressure (no prediction), most packets roughly experience $(\log(V))^2$ delay. Thus, with a moderate size prediction window size, the server can serve most packets before they enter the system. Since we use a log-scale for the x-axis, we do not plot the fraction for packets that have zero delay. Instead, we show the numbers in the plot. We see that $47.62\%$ of the packets are served before they enter the system, whereas $92.89\%$ of the packets are served for queue $2$. This is expected since queue $2$ has a larger prediction window size. 
The results also demonstrate the power of predictive scheduling. 
\begin{figure}[cht]
\centering
\includegraphics[height=2.3in, width=3.4in]{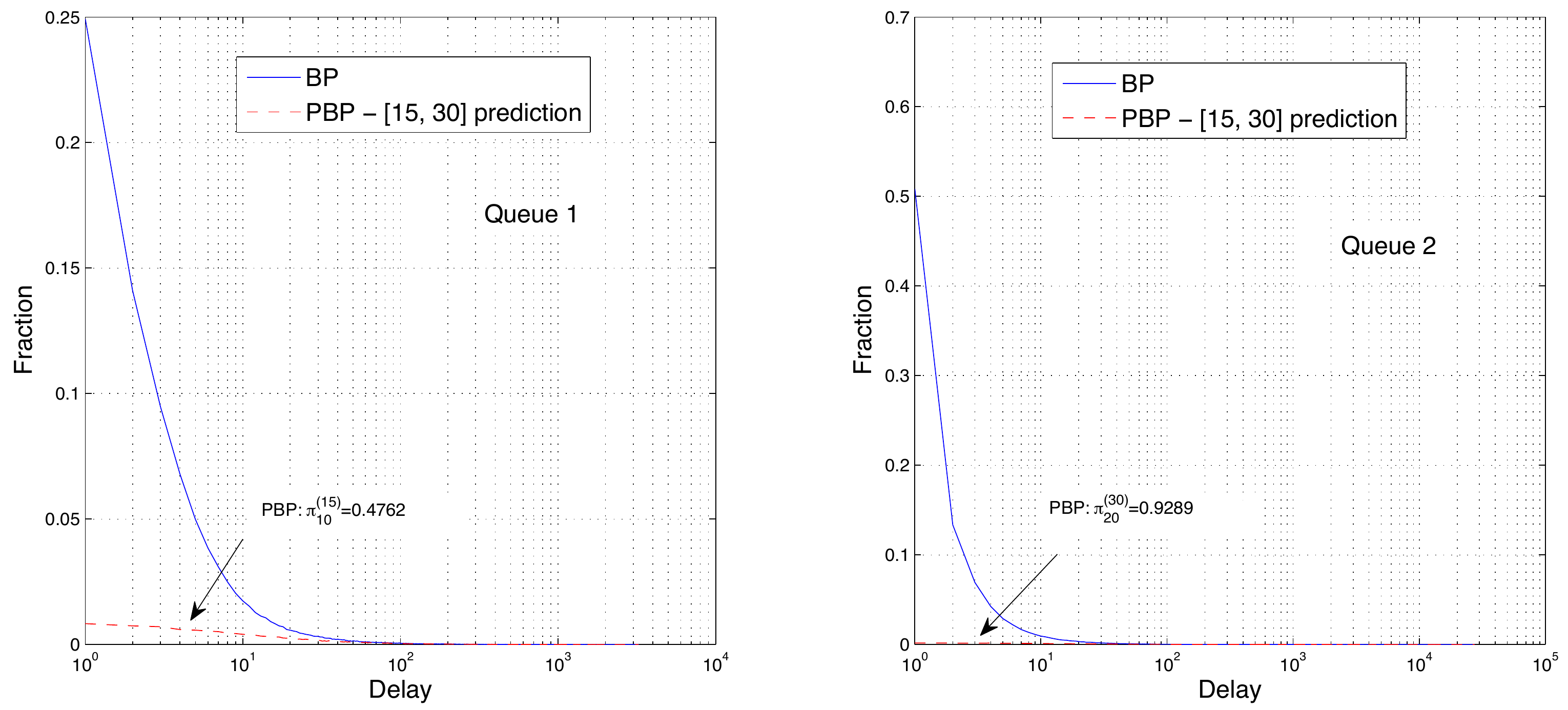}
\vspace{-.1in}
\caption{Packet Delay distribution under \textsf{PBP} with LIFO scheduling ($V=10$ and $\bv{D}=[15, 30]$). We see that a large fraction of the packets now experience $0$ delay! This is because with a moderate size prediction window, most packets are served before they arrive at $Q_n^{(-1)}(t)$. In the plots, $\pi^{(D_n)}_{n0}$ denotes the fraction of packets experience $0$ delay. }\label{fig:dist-lifo}
\end{figure}

\section{Conclusion}\label{section:conclusion}
In this paper, we investigate the fundamental benefit of predictive scheduling in controlled queue systems. Based on a lookahead prediction window model, we  establish a novel system-equivalence result, which enables detailed analysis of the system under predictive scheduling using traditional queueing network control methods. We then  propose the \textsf{Predictive Backpressure (PBP)} algorithm. We show that \textsf{(PBP)} can achieve a cost performance that is arbitrarily close to the optimal, while guaranteeing that the average system delay vanishes as the prediction window size increases. 

\vspace{-.1in}
\section*{Appendix A -- Proof of Theorem \ref{theorem:sample-path}}
Here we prove  Theorem \ref{theorem:sample-path}. 
\begin{proof}
We prove the result by induction with the aid of the following figure showing the evolution of $\hat{Q}_n(t)$. 
\begin{figure}[cht]
\centering
\includegraphics[height=0.4in, width=2.4in]{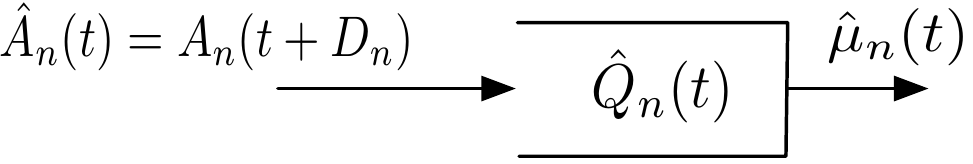}
\caption{The original queue without prediction and with delayed arrivals.}\label{fig:queue-delayed}
\vspace{-.1in}
\end{figure}

First we see that the the result holds for $t=0$. To see this, note that at time $0$,   $\hat{Q}_n(0)=\sum_{t=0}^{D_n-1}A_n(t)$. In the system under predictive scheduling, since $Q^{(0)}_n(0)=0$ and $Q_n^{(d)}(0)=A_n(d)$ for $d\in\{0, ..., D_n-1\}$, we also have $Q^{\text{sum}}_n(0)=\sum_{t=0}^{D_n-1}A_n(t)$. 

Now suppose the result holds for all $t=0, ..., k$, we show that it holds for $t=k+1$. To show this, we label all the packets in $\hat{Q}_n(k)$ starting from the head of the queue. 
That is, the head-of-line packet is labeled $p_1$, and the last packet in the queue is labeled $p_{\hat{Q}_n(k)}$. 

Since at time $k$, we have $\hat{Q}_n(k)=Q_n^{\text{sum}}(k)$, we also label the packets in $\{Q_n^{(d)}(k)\}_{d=-1}^{D_n-1}$ starting from the right in Fig. \ref{fig:queue}. That is, we label the HOL packet in $Q_n^{(-1)}(k)$ as $p'_1$ and the last packet in $Q_n^{(-1)}(k)$ as $p'_{Q_n^{(-1)}(k)}$. Similarly, we label the packets in $Q_n^{(d)}(k)$ with labels $p'_{\sum_{d'<d}Q_n^{d'}(k)+1}$ to $p'_{\sum_{d'\leq d}Q_n^{d'}(k)}$ for all $d$. 

Using the queueing dynamic equation (\ref{eq:q-4}), we know that in time slot $k$, $\tilde{\mu}_n(k)=\min[\mu_n(k), \hat{Q}_n(k)]$ packets will be served from $\hat{Q}_n(k)$. Now consider the queues $\{Q_n^{(d)}(k)\}_{d=-1}^{D_n-1}$. We see that since the scheduling policy is fully-efficient, we must have that the number of packets served from these queues is also $\tilde{\mu}_n(k)=\min[\mu_n(k), \hat{Q}_n(k)]$. 
To see this, note that if $\tilde{\mu}_n(k)=\mu_n(k)$, then it means that there are more packets in the queues than the number of packets that can be served. In this case, we must have $\mu_n^{(d)}(k)\leq Q_n^{(d)}(k)$ for all $d$. Also, because the policy is fully-efficient, $\sum_d\mu_n^{(d)}(k)=\mu_n(k)$. 
Hence, exactly $\mu_n(k)$ packets will be served from $\{Q_n^{(d)}(k)\}_{d=-1}^{D_n-1}$, resulting in $\hat{Q}_n(k+1)=Q_n^{\text{sum}}(k+1)$. On the other hand, suppose $\tilde{\mu}_n(k)=\hat{Q}_n(k)$. 
This means that there is enough service opportunities to clear the awaiting packets. In this case, since the scheduling policy is fully-efficient,  exactly $\hat{Q}_n(k)$ packets will be served. Thus,  in both cases,  we have $\hat{Q}_n(k+1)=Q_n^{\text{sum}}(k+1)=A_n(k+D_n)$. This completes the proof of Theorem  \ref{theorem:sample-path}. 
\end{proof}

\vspace{-.1in}
\section*{Appendix B -- Proof of Theorem \ref{theorem:dist-shift}}
Here we prove  Theorem \ref{theorem:dist-shift}. 
\begin{proof}
From Theorem \ref{theorem:sample-path}, we see that $Q_n^{\text{sum}}(t) = \hat{Q}_n(t)$ for all time. Hence, if the two queueing systems use the same queueing discipline in choosing what packets to serve, then every packet will experience the exact same delay in $\hat{Q}_n(t)$ and in $\{Q_n^{(d)}(t)\}_{d=-1}^{D_n-1}$. 

However, in $Q_n^{\text{sum}}(t)$, a packet will enter the actual system only after spending one unit of time in each of the queues in $\{Q_n^{(d)}(k)\}_{d=0}^{D_n-1}$, which takes exactly $D_n$ slots in total. Thus, any packet experiencing a $k$-slot delay will experience $[k-D_n]^+$  delay in $Q_n^{(-1)}(t)$. This completes the proof of Theorem \ref{theorem:dist-shift}. 
\end{proof}

\vspace{-.1in}
\section*{Appendix C -- Proof of Theorem \ref{theorem:delay-red}}
We prove Theorem \ref{theorem:delay-red} here. 
\begin{proof}
Using Corollary \ref{coro:dist-shift}, we see that in the predictive system, the average system backlog size is given by: 
\begin{eqnarray}
N^{\text{P}}_{\text{tot}} = \sum_{n=1}^N\lambda_n \sum_{k\geq0}k\pi_{nk+D_n}. \label{eq:q-predict}
\end{eqnarray}
On the other hand, the average system backlog without prediction is given by: 
\begin{eqnarray}
N_{\text{tot}} = \sum_{n=1}^N\lambda_n \sum_{k\geq0}k\pi_{nk}. \label{eq:q-original}
\end{eqnarray}
Using (\ref{eq:q-original}) and (\ref{eq:q-predict}), we conclude that: 
\begin{eqnarray} 
N_{\text{tot}} - N^{\text{P}}_{\text{tot}}  = \sum_{n=1}^N\lambda_n\bigg(  \sum_{1\leq k\leq D_n} k \pi_{nk} + D_n \sum_{k\geq1} \pi_{nk+D_n}  \bigg). 
\end{eqnarray}
Using Little's theorem and dividing both sides by $\sum_n\lambda_n$, we see that (\ref{eq:delay-reduction}) follows. 

Now we prove (\ref{eq:limit-r}). By taking a limit as $\bv{D}\rightarrow\infty$, we first obtain: 
\begin{eqnarray}
\hspace{-.2in}&&\lim_{\bv{D}\rightarrow\infty}  \sum_{n}\lambda_n\sum_{1\leq k\leq D_n} k \pi_{nk} = N_{\text{tot}}. 
\end{eqnarray}
Then, using the fact that $W_{\text{tot}}<\infty$, we have: 
\begin{eqnarray}
\lim_{D_n\rightarrow\infty} D_n \sum_{k\geq1} \pi_{nk+D_n} = 0, \,\,\forall\,\,n. 
\end{eqnarray}
Using the above in (\ref{eq:delay-reduction}), we see that the corollary follows. 
\end{proof}

\section*{Appendix D -- Proof of Theorem \ref{theorem:opt-remains}}
In this section, we prove Theorem \ref{theorem:opt-remains} using a similar argument as in \cite{neelyenergy}. 
\begin{proof}
We first see that since any policy without prediction is also a  feasible policy for the predictive system, $f_{\text{av}}^{\bv{D}*} \leq f_{\text{av}}^{*}$ by definition. 

We now prove that $f_{\text{av}}^{\bv{D}*} \geq f_{\text{av}}^{*}$. 
Consider any predictive scheduling scheme $\Pi_P$ that ensures system stability.  Consider the set of slots $t\in\{0, ..., M\}$. Let $\script{T}_{s_i}(M)$ denote the set of slots with $\bv{S}(t)=s_i$ and let $T_{s_i}(M)$ denote its cardinality. We also define the conditional empirical average of transmission rate and power cost as follows: 
\begin{eqnarray}
\hspace{-.3in}&&(\overline{\mu}^{(s_i)}_1(M), ..., \overline{\mu}^{(s_i)}_N(M), \overline{f}^{(s_i)}(M))\\
\hspace{-.3in}&&\quad \triangleq\sum_{t\in \script{T}_{s_i}(M) } \frac{(\mu_1(s_i, \bv{P}(t)), ..., \mu_N(s_i, \bv{P}(t)), f(s_i, \bv{P}(t)) ) }{  T_{s_i}(M) }. \nonumber
\end{eqnarray}
The above is indeed a mapping from the $N$-dimensional power vector space into the $N+1$ dimensional space, and that the right-hand-side is  a convex combination fo the points in the $N+1$ dimensional space. Hence, using an argument based on Caratheodory's theorem as in \cite{neelyenergy}, one can show that for every $M$, there exists probabilities $\{a_m^{(s_i)}(M)\}_{m=1}^{N+2}$ and power allocation vectors $\{\bv{P}_m^{(s_i)}(M)\}_{m=1}^{N+2}$ such that: 
\begin{eqnarray*}
\overline{\mu}_n^{(s_i)}(M) &=& \sum_{m=1}^{N+2}a_m^{(s_i)}(M)\mu_n(s_i, \bv{P}^{(s_i)}_m(M)), \,\,\forall\,n, \\
\overline{f}^{(s_i)}(M) & = & \sum_{m=1}^{N+2}a_m^{(s_i)}(M)f(s_i, \bv{P}^{(s_i)}_m(M)). 
\end{eqnarray*}
Now  define: 
\begin{eqnarray*}
\hspace{-.3in}&&(\overline{\mu}_1(M), ..., \overline{\mu}_N(M), \overline{f}(M))\\
\hspace{-.3in}&&\qquad\qquad \triangleq \sum_{s_i}\frac{ T_{s_i}(M) }{M} (\overline{\mu}^{(s_i)}_1(M), ..., \overline{\mu}^{(s_i)}_N(M), \overline{f}^{(s_i)}(M)). 
\end{eqnarray*}
Using the ergodicity of the channel state process, the continuity of $f(s_i, \bv{P}(t))$ and $\mu_n(s_i, \bv{P}(t))$, and the compactness of $\script{P}^{(s_i)}$, one can find a sequence of times $\{M_i\}_{i=1}^{\infty}$ and a set of limiting probabilities $\{a_m^{(s_i)}\}_{m=1}^{N+2}$ and power vectors $\{ \bv{P}_m^{(s_i)}\}_{m=1}^{N+2}$ such that: 
\begin{eqnarray}
f_{\text{av}}^{\Pi_P} &=& \sum_{s_i}\pi_{s_i}\sum_{m=1}^{N+2}a_m^{(s_i)}f(s_i, \bv{P}^{(s_i)}_m),\label{eq:cost-timeshare}\\
\overline{\mu}_{n}^{\Pi_P} &=& \sum_{s_i}\pi_{s_i}\sum_{m=1}^{N+2}a_m^{(s_i)}\mu_n(s_i, \bv{P}^{(s_i)}_m), \,\,\forall\,\, n. \label{eq:rate-timeshare}
\end{eqnarray}
Here $f_{\text{av}}^{\Pi_P}$ denotes the average cost under the scheme $\Pi_P$ and $\overline{\mu}_{n}^{\Pi_P}$ denotes the average \emph{total} allocated transmission rate to queue $n$ under $\Pi_P$. This shows that the average cost and the average allocated rate to any queue under a predictive scheme can be achieved by some randomized schemes. 

Now consider any user $n$. Let $\beta_n^{(d)}(t)$ be the number of packets that enter  $Q_n^{(d)}(t)$ at time $t$ and let $\mu_n^{(d)}(t)$ denote the service rate allocated to serve the packets in $Q_n^{(d)}(t)$ at time $t$. Further let $\eta_n^{(d)}(t)$ be the number of packets served from $Q_n^{(d)}(t)$ at time $t$. 
Then, denote $\beta_n^{(d)}$, $\mu^{(d)}_n$, and $\eta_n^{(d)}$ their average values, i.e., 
$\beta_n^{(d)} =  \lim_{T\rightarrow\infty}\frac{1}{T}\sum_{t=0}^{T-1} \expect{\beta_{n}^{(d)}(t)}$, $\mu_n^{(d)} =  \lim_{T\rightarrow\infty}\frac{1}{T}\sum_{t=0}^{T-1} \expect{\mu_{n}^{(d)}(t)}$, $\eta_n^{(d)} =  \lim_{T\rightarrow\infty}\frac{1}{T}\sum_{t=0}^{T-1} \expect{\eta_{n}^{(d)}(t)}$. \footnote{Here we assume these limits exist. Note that since $A_n(t)$, $\eta_n^{(d)}(t)$ and $\mu_n^{(d)}(t)$ are all bounded, they are equal to the sample path limits with probability $1$. }
 
Using the queueing dynamics of $\{Q_n^{(d)}(t)\}_{d=-1}^{D_n-1}$, we have: 
\begin{eqnarray}
\beta_n^{(d)} - \beta_n^{(d-1)} = \eta_n^{(d)}, \,\,\forall\, d=D_n-1, ..., 0. \label{eq:pre-queue-ineq}
\end{eqnarray}
Because $\eta_n^{(d)}(t)\leq \mu_n^{(d)}(t)$ and $\sum_{d}\mu_n^{(d)}(t)=\mu_n(t)$ for all time, we have:  
\begin{eqnarray}
\eta_n^{(d)} \leq \mu_n^{(d)}, \,\,\sum_d\mu_n^{(d)}\leq\overline{\mu}_n^{\Pi_P}.  \label{eq:rate-ineq}
\end{eqnarray}
Since the system is stable, i.e., $Q^{(-1)}_n(t)$ is stable, we must have: 
\begin{eqnarray}
\beta_n^{(-1)}\leq\eta_n^{(-1)}\leq\mu_n^{(-1)}. \label{eq:q0-ineq}
\end{eqnarray}
Summing (\ref{eq:q0-ineq}) and (\ref{eq:pre-queue-ineq}) over $d=-1, ..., D_n-1$, using (\ref{eq:rate-ineq}), and using $\beta_n^{D_n-1}=\lambda_n$, we conclude that: 
\begin{eqnarray*} 
\lambda_n\leq\sum_{d=-1}^{D_n-1}\eta_n^{(d)}\leq \sum_{d=-1}^{D_n-1}\mu_n^{(d)}\leq \overline{\mu}_n^{\Pi_P}. 
\end{eqnarray*}
This shows that for any stabilizing predictive policy, one can find an equivalent stationary and randomized scheduling policy, which results in the same cost that can be expressed as (\ref{eq:opt-noprediction}), and generates the same service rates that must satisfy the constraint (\ref{eq:cond}). Since $f_{\text{av}}^*$ is defined to be the minimum cost over the entire class of such stationary and randomized schemes, we conclude that $f_{\text{av}}^{\bv{D}*}\geq f_{\text{av}}^*$. 
\end{proof}

\vspace{-.1in}
 \section*{Appendix E -- Proof of Theorem \ref{theorem:pbp}}
 Here we prove Theorem \ref{theorem:pbp}. 
 \begin{proof}
 To prove the results, we consider an auxiliary system that has the exact same setting,  and the same arrival and channel state processes, except that the arrival process for queue $n$ is  given by $\tilde{A} _n(t)=A_n(t+D_n)$ and the initial queue size is given by $\tilde{Q}_n(0)=\sum_{t=0}^{D_n-1}A_n(t)$ for all $n$. Here $\tilde{Q}_n(t)$ denotes the size of queue $n$ in this auxiliary system. Note here $\tilde{Q}_n(0)\leq D_nA_{\max}$. 
 
Then, applying the \textsf{BP} algorithm with FIFO to this system and using Theorem \ref{theorem:sample-path}, we see that at every time instance, 
 \begin{eqnarray}
\tilde{Q}_n(t) = Q^{\text{sum}}_n(t). \label{eq:q-equal}
\end{eqnarray}
Therefore, the \textsf{BP} algorithm in the auxiliary system will choose the exact same control actions as \textsf{PBP} in the actual system. Since both systems have the  same arrival and channel state processes, the two systems will evolve identically. 
Thus, the average power cost and the average queue size will be the same in both systems. Hence, Theorem  \ref{theorem:pbp} follows from Theorem \ref{theorem:bp}. 
 \end{proof}

\section*{Appendix F -- Proof of Theorem \ref{theorem:pbp-delay}}
We prove Theorem \ref{theorem:pbp-delay}. 
\begin{proof}
We prove the results with Little's theorem.  The main idea is to show that the average system queue length is roughly reduced by $\sum_n\lambda_n D_n$. To prove this, we will show that the average total service rate allocated to the prediction queues is $O(\frac{1}{V})$. Then, the average rate of the packets that go through $\{Q_n^{(d)}(t)\}_{d=1}^{D_n}$ will roughly be $\lambda_n$, and so the average queue size is reduced by $\sum_n\lambda_nD_n$. 

First, using (\ref{eq:q-equal-thm}) and (\ref{eq:prob_pmr_special}), we see that in steady state, 
\begin{eqnarray*}
\prob{|Q_n^{\text{sum}}(t) - \gamma_{n}^*|>G+Km}\leq ce^{-m}. 
\end{eqnarray*}
Using the fact that $Q_n^{\text{sum}}(t) = \sum_{d=-1}^{D_n-1}Q_n^{(d)}(t)$, we have: 
\begin{eqnarray*}
\prob{ Q_n^{(-1)}(t) < \gamma_{n}^*- G-Km - \sum_{d=0}^{D_n-1}Q_n^{(d)}(t)}\leq ce^{-m}. 
\end{eqnarray*}
Now let $m=(\log(V))^2$. Since $\gamma_{n}^*=\Theta(V)$, we see that when $V$ is sufficiently large, we have: 
\begin{eqnarray}
\hspace{-.3in}&&\gamma_{n}^*- G-Km - \sum_dQ_n^{(d)}(t)\nonumber\\  
\hspace{-.3in}&&\quad = \Theta(V) - G-K(\log(V))^2 -  \sum_dQ_n^{(d)}(t) \nonumber\\
\hspace{-.3in}&&\quad\stackrel{(a)}{\geq} \Theta(V) - G-K(\log(V))^2 - D_nA_{\max} \nonumber\\
\hspace{-.3in}&&\quad\stackrel{(b)}{\geq} \mu_{\max}. \label{eq:bdd-0}
\end{eqnarray}
In (a) we use  the fact that $Q_n^{(d)}(t)\leq A_{\max}$ for all $1\leq d\leq D_n$, and in (b) we use the fact that $V$ is sufficiently large and $D_n=O(\frac{1}{A_{\max}}[\gamma_{n}^*-G-K(\log(V))^2 - \mu_{\max}]^+)$ for all $n$. This shows that the probability for $Q^{(-1)}_n(t)$ to go below $\mu_{\max}$ is at most $ce^{-(\log(V))^2}=\frac{c}{V^{\log(V)}}$. 

Using the fact that under the FIFO queueing discipline, a prediction queue $Q_n^{(d)}(t)$ will be served only when $Q^{(-1)}_n(t)<\mu_{\max}$. We conclude that the average service rate allocated to the prediction queues is no more than $\frac{c\mu_{\max}}{V^{\log(V)}}$. Hence, the average traffic rate of the packets that traverse all prediction queues and eventually enter $Q^{(-1)}_n(t)$ is at least $[\lambda_n - \frac{c\mu_{\max}}{V^{\log(V)}}]^+$. Since every packet stays $1$ slot in every prediction queue. Using Little's theorem, we conclude that the average size of the prediction queues, denoted by $\sum_{d=0}^{D_n-1}\overline{Q}_n^{(d)}$ satisfies: $\sum_{d=0}^{D_n-1}\overline{Q}_n^{(d)}\geq [\lambda_n - \frac{c\mu_{\max}}{V^{\log(V)}}]^+D_n$. Hence, (\ref{eq:q-av-0}) follows. 
\end{proof}

$\vspace{-.1in}$
\bibliographystyle{unsrt}
\bibliography{../mybib}

\end{document}